\documentclass[12pt,a4paper]{article}

\usepackage[margin=2cm]{geometry}
\usepackage{t1enc}
\usepackage[utf8]{inputenc}
\usepackage{amsthm,amsmath,amssymb}
\usepackage{graphicx}
\usepackage{enumerate}
\usepackage{hyperref}
\usepackage{bm}

\usepackage{amsfonts}
\usepackage{graphicx}
\usepackage{bm}
\usepackage{amsmath, amsthm, amssymb}
\usepackage{graphicx}
\usepackage{hyperref}
 \usepackage{relsize}

\usepackage{bbm}

\theoremstyle{plain}

\newtheorem{theorem}{Theorem}
\newtheorem*{theorem-non}{Theorem}
\newtheorem{lemma}[theorem]{Lemma}

\theoremstyle{definition}

\newcommand{\Addresses}{{
  \bigskip
  \footnotesize

  Andr\'as M\'esz\'aros, 

\textsc{Central European University, Budapest \&}\par\nopagebreak
    
\textsc{Alfr\'ed R\'enyi Institute of Mathematics, Budapest}\par\nopagebreak
  \textit{E-mail address:}: \texttt{Meszaros\_Andras@phd.ceu.edu}

}}

\title{On the free energy density of  factor models \\ on biregular graphs}
\author{Andr\'as M\'esz\'aros}

\begin{document}

\maketitle

\begin{abstract}
Let $h(0),h(1),\dots,h(k)$ be a symmetric concave sequence. For a $(d,k)$-biregular factor graph $G$ and  $x\in \{0,1\}^V$, we define the Hamiltonian
\[H_G(x)=\sum_{f\in F} h\left(\sum_{v\in \partial f} x_v\right),\]    
where $V$ is the set of variable nodes, $F$ is the set of factor nodes. We prove that if $(G_n)$ is a large girth sequence of  $(d,k)$-biregular factor graphs, then the  free energy density  of $G_n$ converges. The limiting free energy density is given by the Bethe-approximation. 
\end{abstract}

\section{Introduction}

Let $G$ be a bipartite graph with color classes $V$ and $F$. Nodes in $V$ are called variable nodes,  nodes in $F$ are called factor nodes. Fix two integers $d,k\ge 2$. Assume that all the variable nodes have degree $d$, and all the  factor nodes have degree $k$.  Graphs satisfying the conditions above will be called $(d,k)$-biregular factor graphs. 
 
Let $h_0,h_1,\dots,h_k$ be a symmetric concave sequence, that is,
\begin{align*}
h_i&=h_{k-i} &\forall 0\le i\le k,\\
2h_i&\ge h_{i-1}+h_{i+1} &\forall 1\le i\le k-1.
\end{align*}
We use the notations $h_i$ and $h(i)$ interchangeably.

We define the Hamiltonian $H_G:\{0,1\}^V\to\mathbb{R}$ as follows. For $x\in \{0,1\}^V$, we set
\[H_G(x)=\sum_{f\in F} h\left(\sum_{v\in \partial f} x_v\right).\] 
Here $\partial f$ denotes the set of neighbors of $f$. 

Given a non-negative real parameter $\beta$ called the \emph{inverse temperature}, we define the \emph{partition function} $Z_G(\beta)$ by
\[Z_G(\beta)=\sum_{x\in \{0,1\}^V} \exp\left(-\beta H_G(x)\right).\]

We also define  the \emph{free energy density} $\Phi_G(\beta)$ as
\[\Phi_G(\beta)=\frac{\log Z_G(\beta)}{|V|}.\]

A sequence of finite graphs has large girth if the length of the shortest cycle tends to infinity. 
\begin{theorem}\label{TheoremFO}
Let $G_1,G_2,\dots$ be a large girth sequence of $(d,k)$-biregular factor graphs. Then
\[\lim_{n\to\infty}\Phi_{G_n}(\beta)=\Phi(\beta),\]
where $\Phi(\beta)$ is specified below.
\end{theorem}

Using the same potentials as above, one can define the set of Gibbs measures $\mathcal{G}$ on the \break {$(d,k)$-biregular} infinite tree $\mathbb{T}$. Since $\mathbb{T}$ is a tree, all the extremal elements of $\mathcal{G}$ are Markov-chains. Thus, they can be described by an assignment of messages to each directed edge of $\mathbb{T}$ such that these messages form a BP-fixed point. We will be particularly interested in BP-fixed points that are symmetric in the sense that for all directed edges of $\mathbb{T}$ that are directed from variable nodes to factor nodes, the messages are all equal. Using the so called log-likelihood parametrization, we will see that these symmetric BP-fixed points correspond to the zeros of the function  

\begin{align*}
F(\beta,t)&=-t+(d-1)\log\left(\frac{\sum_{i=0}^{k-1} {{k-1}\choose{i}} \exp(-\beta h_{i+1}+it)}{\sum_{i=0}^{k-1} {{k-1}\choose{i}} \exp(-\beta h_{i}+it)}\right).
\end{align*}

See Subsection \ref{subsecT} for a more detailed explanation. The next lemma characterizes the zero set of $F$.

\begin{lemma}\label{lemmafix}
There is a $\beta_c\in (0,\infty]$, and a differentiable monotone increasing function \break $\mathfrak{g}:(\beta_c,\infty)\to \mathbb{R}_{>0}$ with $\lim_{\beta \to \beta_c} \mathfrak{g}(\beta)=0$, such that

\begin{align*}
\{(\beta,t)\in \mathbb{R}_{\ge 0}\times \mathbb{R}\quad |\quad F(\beta,t)=0\}=&\mathbb{R}_{\ge 0}\times \{0\}\\ 
&\cup\{(\beta,+\mathfrak{g}(\beta))\quad|\quad \beta\in(\beta_c,\infty)\}\\
&\cup\{(\beta,-\mathfrak{g}(\beta))\quad|\quad \beta\in(\beta_c,\infty)\}.
\end{align*}
\end{lemma}

We extend the function $\mathfrak{g}$  by setting $\mathfrak{g}(\beta)=0$ for all $\beta\le \beta_c$. 

Note that the zero $(\beta,\mathfrak{g}(\beta))$ corresponds to the extremal Gibbs-measure $\mu_1$ with all $1$ boundary conditions.  Similarly, the zero  $(\beta,-\mathfrak{g}(\beta))$ corresponds to the extremal Gibbs-measure $\mu_0$ with all $0$ boundary conditions. In particular, they coincide if and only if $\beta\le \beta_c$. It will be crucial later that we have a first order phase transition, that is, $\mathfrak{g}$ is continuous at $\beta_c$.

For  $(\beta,t)$ such that $F(\beta,t)=0$, we define the Bethe-approximation
\[\Phi(\beta,t)=- d\cdot t\cdot t_v +\frac{d}{k}\log \left(\sum_{i=0}^k {{k}\choose{i}} \exp(-\beta h_i+it)\right) -(d-1) H_2(t_v),\]
where $H_2(x)=-x\log x-(1-x)\log(1-x)$, and 
\[t_v=\frac{\exp(\frac{d}{d-1}t)}{1+\exp(\frac{d}{d-1}t)}.\]

\begin{theorem}\label{TheoremFO2}
For the limiting constant in Theorem \ref{TheoremFO}, we have
\[\Phi(\beta)=\Phi(\beta,\mathfrak{g}(\beta)).\] 
\end{theorem}


We prove that $\limsup_{n\to\infty} \Phi_{G_n}(\beta)\le\Phi(\beta,\mathfrak{g}(\beta))$ by showing that \[\limsup_{n\to\infty} (\Phi_{G_n}(\beta))'\le(\Phi(\beta,\mathfrak{g}(\beta)))'\] for any $\beta\neq \beta_c$. 

For a $(d,k)$-biregular factor graph $G$ and a $\beta\ge 0$, we define the Gibbs probability measure $\mu_G^\beta$ on $\{0,1\}^V$ by setting
\[\mu_G^\beta(x)=\frac{\exp(-\beta H_G(x))}{Z_G(\beta)}.\] 

Let $X\in \{0,1\}^V$ a random vector with law $\mu_G^\beta$, and let $a$ be a uniform random factor node independent from $X$. Then
\begin{equation*}
\Phi_G'(\beta)=\frac{d}{k}\mathbb{E}\left(-h\left(\sum_{i\in \partial a} X_i\right)\right),
\end{equation*}
where the expectation is over the random choice of $X$ and $a$. 

 Taking an appropriate subsequential weak limit of the Gibbs measures $\mu_{G_n}^\beta$ one can obtain a Gibbs measure $\mu$ on the $(d,k)$-biregular infinite tree $\mathbb{T}$ with the following property. Let $a$ be a fixed variable node of $\mathbb{T}$, let $X\in \{0,1\}^V$ a random vector with law $\mu$, then
\[\limsup_{n\to\infty} \Phi_{G_n}'(\beta)=\frac{d}{k}\mathbb{E}\left(-h\left(\sum_{i\in \partial a} X_i\right)\right),\] 
where the expectation over a random choice of $X$. 

The next lemma is a  crucial ingredient of the proof.  

\begin{lemma}\label{lemma4}
For a fixed $\beta$, let $\mu$ be a Gibbs measure on $\mathbb{T}$. Fix a factor node $a$. Let $X$ be a random element of $\{0,1\}^V$ with law $\mu$. Let $X^{(1)}$ and $X^{(0)}$ be  random elements of $\{0,1\}^V$ with law $\mu_1$ and $\mu_0$. Then
\[\mathbb{E} h\left(\sum_{i\in\partial a} X_i\right) \ge \mathbb{E}h\left(\sum_{i\in\partial a} X_i^{(1)}\right)=\mathbb{E}h\left(\sum_{i\in\partial a} X_i^{(0)}\right).\]

In other words, among all Gibbs-measures, the local expected energy at a factor node is minimized by the Gibbs-measure with all $1$ or all $0$ boundary conditions.
\end{lemma}

We will later see that
\[\frac{d}{k}\mathbb{E}h\left(\sum_{i\in\partial a} X_i^{(1)}\right)=-(\Phi(\beta,\mathfrak{g}(\beta)))'.\]
Thus, putting everything together, we get the desired inequality \[\limsup_{n\to\infty} (\Phi_{G_n}(\beta))'\le(\Phi(\beta,\mathfrak{g}(\beta)))'\] for any $\beta\neq \beta_c$. See Section \ref{putt} for more details. 

This method that we use to obtain an upper bound on the asymptotic free energy density is called interpolation scheme. This method has been applied by Dembo, Montanari, Sun~\cite{dembo2013factor} and Sly, Sun \cite{sly2014counting} to obtain similar results. 

To prove that  $\liminf_{n\to\infty} \Phi_{G_n}(\beta)\ge\Phi(\beta,\mathfrak{g}(\beta))$, we will use an inequality of Ruozzi \cite{ruozzi2012bethe}. In fact this inequality provides us the even stronger statement that $\Phi_{G_n}(\beta)\ge\Phi(\beta,\mathfrak{g}(\beta))$ for all $n$. 

\textbf{The structure of the paper.} In Section \ref{sec2}, we recall the notion of BP fixed points and their relation to extremal Gibbs-measures. We  calculate the derivative of the Bethe-approximation with respect to the inverse temperature. We also state a few well-known correlation inequalities. In Section \ref{sec3}, we characterize the zero set of $F$. In Section \ref{sec4}, we prove Lemma \ref{lemma4}. Finally, Section~\ref{putt} contains the proof of Theorem \ref{TheoremFO} and Theorem \ref{TheoremFO2}.
 
\textbf{Acknowledgements.} The author is grateful to P\'eter  Csikv\'ari and Mikl\'os Ab\'ert  for their comments. The author was partially
supported by the ERC Consolidator Grant 648017.

  
\section{Preliminaries}\label{sec2}

\subsection{BP fixed point on finite trees}

For a more detailed introduction to belief propagation, we refer the reader to Chapter 14 of the book by M\'ezard and Montanari \cite{mezard2009information}. We also try to follow their notations in this subsection.

 For our limited purposes in this subsection, a factor graph will mean a triple $(G,V,F)$, where $G$ is a finite bipartite graph, $(V,F)$ is a proper two coloring of the nodes of $G$, moreover, we also assume that each vertex in $F$ has degree $k$. As before, nodes in $V$ are called variable nodes, and nodes in $F$ are called factor nodes. Given a factor node $a$ and $x_{\partial a}\in \{0,1\}^{\partial a}$, we define
\[\psi_a(x_{\partial a})=\exp\left(-\beta h\left(\sum_{i\in \partial a} x_i\right) \right),\]
where $\beta\ge 0$ is a fixed parameter. With each factor graph $(G,V,F)$, we associate a random vector $x\in \{0,1\}^V$, by setting
\[p(x)=\frac{1}{Z} \prod_{a\in F} \psi_a(x_{\partial a})\]
for all $x\in \{0,1\}^V$. Here $Z$ is an appropriate normalizing constant. 

Let $ai$ be an edge of $G$ such that $a\in F$ and $i\in V$. 

We define $\left(\nu_{i\to a}(x_i)\right)_{x_i\in \{0,1\}}$, as the marginal of $x_i$ in the random vector corresponding to the factor graph obtained from $G$ by deleting the factor node $a$.

We define $\left(\hat{\nu}_{a\to i}(x_i)\right)_{x_i\in \{0,1\}}$, as the marginal of $x_i$ in the random vector corresponding to the factor graph obtained from $G$ by deleting factor nodes in $\partial i\backslash a$.

For the proof of the next lemma, see for example \cite[Theorem 14.1]{mezard2009information}.
\begin{lemma} \label{lemmameseq}
If $G$ is  a forest, then  the messages $\nu_{i\to a}(x_i)$ and $\hat{\nu}_{i\to a}(x_i)$ are uniquely determined by the following  equations. For each edge $ai$ of $G$ such that $a\in F$ and $i\in V$, and $x_i\in \{0,1\}$, we have
\begin{align*}
\nu_{i\to a}(x_i)&=\frac{\prod_{b\in \partial i\backslash a} \hat{\nu}_{b\to i}(x_i)}{\sum_{y_i}\prod_{b\in \partial i\backslash a} \hat{\nu}_{b\to i}(y_i)},\\
\hat{\nu}_{a\to i}(x_i)&=\frac{\sum_{x_{\partial a\backslash i}} \psi_a(x_{\partial a})\prod_{j\in \partial a\backslash i} {\nu}_{j\to a}(x_i)}{\sum_{y_{\partial a}} \psi_a(y_{\partial a})\prod_{j\in \partial a\backslash i} {\nu}_{j\to a}(y_i)}.
\end{align*}
\end{lemma}

Let $F_R$ be a subset of function nodes, $V_R$ be the subset of variable nodes adjacent to $F_R$, and $R=F_R\cup V_R$. Assume that the subgraph of $G$ induced by $R$ is a connected. Then we have the following formula for the marginal $x_{V_R}$, provided that $G$ is a forest.

\begin{equation}\label{eqmarg}
p(x_{V_R})=\frac{1}{Z}\prod_{a\in F_R} \psi_a(x_{\partial a}) \prod_{a\in \partial R} \hat{\nu}_{a\to i(a)}(x_{i(a)}), 
\end{equation} 
where $i(a)$ is the unique neighbor of $a$ in $V_R$. For a proof, see \cite[Section 14.2.3]{mezard2009information}.

It will be more convenient to use the so called log-likelihood ratio to parametrize the messages. That is, let
\[u_{i\to a}=\log\frac{\nu_{i\to a}(1)}{\nu_{i\to a}(0)}\quad\text{ and }\quad\hat{u}_{a\to i}=\log\frac{\hat{\nu}_{a\to i}(1)}{\hat{\nu}_{a\to i}(0)}.\] 

With these notations, the equations in Lemma \ref{lemmameseq} become
\begin{align}
u_{i\to a}&=\sum_{b\in \partial i\backslash a} \hat{u}_{b\to i},\label{ueq1}\\
\hat{u}_{a\to i}&=\log\frac{\sum_{x_{\partial a}, x_i=1 } \psi_a(x_{\partial_a}) \exp\left(\sum_{j\in \partial a\backslash i} x_j u_{j\to a}\right)}{\sum_{x_{\partial a}, x_i=0 } \psi_a(x_{\partial_a}) \exp\left(\sum_{j\in \partial a\backslash i} x_j u_{j\to a}\right)}\label{ueq2}.
\end{align}
Moreover, Equation \eqref{eqmarg} becomes
\[p(x_R)=\frac{1}{Z}\prod_{a\in F_R} \psi_a(x_{\partial a}) \exp\left(\sum_{a\in \partial R} x_{i(a)}\hat{u}_{a\to i(a)}\right).\]
Note that the normalizing constant $Z$ can be different from the constant in Equation \eqref{eqmarg}.

\subsection{BP fixed points on the $(d,k)$-biregular infinite tree}\label{subsecT}

Let $\mathbb{T}$ be an infinite connected tree with a proper two coloring $(V,F)$, such that the nodes in $V$ have degree $d$, and the nodes in $F$ have degree $k$. We call $\mathbb{T}$  the $(d,k)$-biregular infinite tree. To each edge $ai$ of $\mathbb{T}$ with $a\in F$ and $i\in V$, we assign two real numbers $u_{i\to a}$ and $\hat{u}_{a\to i}$. We say the these numbers from a BP fixed point, if they satisfy Equation \eqref{ueq1} and Equation \eqref{ueq2} for all edges.

We will be particularly interested in BP fixed points such that for all edges, we have $u_{i\to a} =t$ for a given $t$.  Such a fixed point exists, if and only if
\begin{equation}\label{BPfixedpointEq}
t=(d-1)\log \frac{\sum_{i=0}^{k-1} {{k-1}\choose{i}}\exp(-\beta h_{i+1}+ti)}{\sum_{i=0}^{k-1} {{k-1}\choose{i}}\exp(-\beta h_{i}+ti)}.
\end{equation}
This motivates the definition of the function $F(\beta,t)$ in the Introduction.     

Now, we shall discuss the analogue of Equation \eqref{eqmarg}. To do this, we need the notion of Gibbs measures. Let $\mu$ be a measure on $\{0,1\}^V$. We say that $\mu$ is a Gibbs measure, if it satisfies the following property for every finite subset $V_R$ of $V$. Let $F_R$ be the set of function nodes that has a neighbor in $V_R$, let $R=V_R\cup F_R$. Note that $\partial R\subset V$. Let $X$ be a random element of $\{0,1\}^V$ with law $\mu$. Then for all $x_{\partial R}\in \{0,1\}^{\partial R}$ and $x_{V_R}\in \{0,1\}^{V_R}$, we have
\[\mathbb{P}(X_{V_R}=x_{V_R}|X_{\partial R}=x_{\partial R})=\frac{1}{Z} \prod_{a\in R_F} \psi_a(x_{\partial a}),\]
where $Z$ only depends on $x_{\partial R}$, but not on $x_{V_R}$.   

A Gibbs measure is called extremal if it can not be written as a non-trivial convex combination of two different Gibbs measures.
\begin{lemma}\label{lemma5}
Let $\mu$ be an extremal Gibbs measure. Then there is a BP fixed point $(u_{i\to a}),(\hat{u}_{a\to i})$ with the following property.
Let $F_R$ be a finite subset of function nodes, $V_R$ be the subset of variable nodes adjacent to $F_R$, and $R=F_R\cup V_R$. Assume that the subgraph of $G$ induced by $R$ is a connected. Let $X$ be a random element of $\{0,1\}^V$ with law $\mu$, then for any $x_{V_R}\in \{0,1\}^{V_R}$, we have
\[\mathbb{P}(X_{V_R}=x_{V_R})=\frac{1}{Z}\prod_{a\in F_R} \psi_a(x_{\partial a}) \exp\left(\sum_{a\in \partial R} x_{i(a)}\hat{u}_{a\to i(a)}\right),\] 
where $Z$ only depend on the choice of $V_R$.
\end{lemma}

For pairwise models this was proved by Zachary \cite{zachary1983countable}. Since our setting is slightly different, we provide a proof in the appendix, although the original proof can be repeated almost word by word.

\subsection{Bethe-approximation}\label{subsecbethe}

Let $G$ be a finite $(d,k)$-biregular factor graph. For $i\in V$, let $\tau_i$ be a probability distribution on $\{0,1\}$. For $a\in F$, let $\tau_a$ be a probability distribution on $\{0,1\}^{\partial a}$. We say that $\tau=\left((\tau_i)_{i\in V},(\tau_a)_{a\in F}\right)$ is in the local marginal polytope $\mathcal{T}$,  if for all $i\in V$ and $a\in V$ such that $ai$ is an edge of $G$, we have
\[\tau_i(x_i)=\sum_{x_{\partial a\backslash i}} \tau_a(x_{\partial a}).\] 

For an element $\tau\in \mathcal{T}$, we define the Bethe approximation at inverse temperature $\beta$ as
\begin{align*}
\Phi_G(\beta,\tau)=&\frac{1}{|V|}\Bigg(-\beta \sum_{a\in F} \sum_{x_{\partial a}} \tau_a(x_{\partial a}) h\left(\sum_{i\in \partial a} x_i\right)\\&-\sum_{a\in F}\sum_{x_{\partial a}}\tau_a(x_{\partial a})\log \tau_a(x_{\partial a}) +(d-1)\sum_{i\in V}\sum_{x_i}\tau_i(x_i)\log \tau_i(x_i)\Bigg) . 
\end{align*}

Let $t=t(\beta)$ be a solution of Equation \eqref{BPfixedpointEq}. With the help of this, we will construct an element $\tau=\tau^{(\beta)}$ of $\mathcal{T}$.   For any $a\in F$, we set
\[\tau_a(x_{\partial a})=\frac{1}{Z} \exp\left(-\beta h\left(\sum_{i\in \partial a} x_i\right)+t\sum_{i\in \partial a} x_i\right ),\]
where 
\[Z=\sum_{i=0}^k {{k}\choose{i}} \exp(-\beta h_i+it).\]  

Let \[t_v=\frac{\exp(\frac{d}{d-1}t)}{1+\exp(\frac{d}{d-1}t)}.\] 

For $i\in V$, we set 
\begin{align*}
\tau_i(0)&=1-t_v,\\
\tau_i(1)&=t_v.
\end{align*}

Then $\tau$ is indeed an element of $\mathcal{T}$. To see this, let $a\in F$, $j\in \partial a$, and set $x_j=1$. Then

\begin{align*}
\sum_{x_{\partial a\backslash j}} \tau_a(x_a)&=\frac{\sum_{i=0}^{k-1} {{k-1}\choose{i}} \exp(-\beta h_{i+1} +(i+1)t)}{\sum_{i=0}^{k} {{k}\choose{i}} \exp(-\beta h_{i} +it)}\\&=\frac{\sum_{i=0}^{k-1} {{k-1}\choose{i}} \exp(-\beta h_{i+1} +(i+1)t)}{\sum_{i=0}^{k-1} {{k-1}\choose{i}} \exp(-\beta h_{i+1} +(i+1)t)+\sum_{i=0}^{k-1} {{k-1}\choose{i}} \exp(-\beta h_{i} +it)}\\&=\left(1+\exp(-t)\left( \frac{\sum_{i=0}^{k-1} {{k-1}\choose{i}} \exp(-\beta h_{i+1} +it)}{\sum_{i=0}^{k-1} {{k-1}\choose{i}} \exp(-\beta h_{i} +it)}\right)^{-1}\right)^{-1}
\\&=\left(1+\exp(-t)\exp\left(-\frac{t}{d-1}\right)\right)^{-1}=t_v=\tau_j(1).
\end{align*}

With this choices of $\tau=\tau^{(\beta)}$, we have
\begin{align*}
{\Phi_G(\beta,\tau)}=&- \frac{d}{k} \frac{\sum_{i=0}^k {{k}\choose{i}} it(\beta)\exp(-\beta h_i+it(\beta))}{\sum_{i=0}^k {{k}\choose{i}} \exp(-\beta h_i+it(\beta))}+\frac{d}{k}\log \sum_{i=0}^k {{k}\choose{i}} \exp(-\beta h_i+it(\beta))\\& -(d-1) H_2(t_v(\beta)),
\end{align*}
where $H_2(x)=-x\log x-(1-x)\log(1-x)$.

Note that
 \begin{align*}
&\frac{\sum_{i=0}^k {{k}\choose{i}} i\exp(-\beta h_i+it(\beta))}{\sum_{i=0}^k {{k}\choose{i}} \exp(-\beta h_i+it(\beta))}
\\&=k\exp(t(\beta))\frac{\sum_{i=0}^{k-1} {{k-1}\choose{i}} \exp(-\beta h_{i+1}+it(\beta))}{\sum_{i=0}^k {{k}\choose{i}} \exp(-\beta h_i+it(\beta))}
\\&=k\exp(t(\beta))\frac{\sum_{i=0}^{k-1} {{k-1}\choose{i}} \exp(-\beta h_{i+1}+it(\beta))}{\sum_{i=0}^{k-1} {{k-1}\choose{i}} \exp(-\beta h_i+it(\beta))+\exp(t(\beta))\sum_{i=0}^{k-1} {{k-1}\choose{i}} \exp(-\beta h_{i+1}+it(\beta))}
\\&=k\exp(t(\beta))\left(\exp(t(\beta))+\left(\frac{\sum_{i=0}^{k-1} {{k-1}\choose{i}} \exp(-\beta h_{i+1}+it(\beta))}{\sum_{i=0}^{k-1} {{k-1}\choose{i}} \exp(-\beta h_i+it(\beta))}\right)^{-1}\right)^{-1}
\\&=\frac{k\exp(t(\beta))}{\exp(t(\beta))+\exp(-\frac{t(\beta)}{d-1})}=k\frac{\exp(\frac{d}{d-1}t(\beta))}{1+\exp(\frac{d}{d-1}t(\beta))}\\&=kt_v.
 \end{align*}

Thus,
\begin{align*}
{\Phi_G(\beta,\tau)}=&- d t(\beta)t_v(\beta) +\frac{d}{k}\log \sum_{i=0}^k {{k}\choose{i}} \exp(-\beta h_i+it(\beta))\\& -(d-1) H_2(t_v(\beta)),
\end{align*}

Note that ${\Phi_G(\beta,\tau)}$ only depends on $t$, but not on the graph $G$. Thus we introduce the notation 
\begin{align*}\Phi(\beta,t(\beta))&=- d t(\beta)t_v(\beta) +\frac{d}{k}\log \sum_{i=0}^k {{k}\choose{i}} \exp(-\beta h_i+it(\beta))\\& -(d-1) H_2(t_v(\beta)).
\end{align*}
 
%

Assume that for each $\beta$, we have a solution  $t(\beta)$ of Equation \eqref{BPfixedpointEq}. Fix a $\beta$ such that derivative $t'(\beta)$ exists. In the rest of this subsection, our aim is to give a formula for $\partial_\beta \Phi(\beta,t(\beta))$. 

Now 
\begin{align*}
\partial_\beta &\frac{d}{k}\log \sum_{i=0}^k {{k}\choose{i}} \exp(-\beta h_i+it(\beta))\\&=
\frac{d}{k}\frac{-\sum_{i=0}^k {{k}\choose{i}} h_i\exp(-\beta h_i+it(\beta))}{\sum_{i=0}^k {{k}\choose{i}} \exp(-\beta h_i+it(\beta))}+t'(\beta)\frac{d}{k}\frac{\sum_{i=0}^k {{k}\choose{i}} it(\beta)\exp(-\beta h_i+it(\beta))}{\sum_{i=0}^k {{k}\choose{i}} \exp(-\beta h_i+it(\beta))}\\&=\frac{d}{k}\frac{-\sum_{i=0}^k {{k}\choose{i}} h_i\exp(-\beta h_i+it(\beta))}{\sum_{i=0}^k {{k}\choose{i}} \exp(-\beta h_i+it(\beta))}+dt'(\beta)t_v(\beta),
\end{align*}

Note that
$H_2'(x)=-\log\frac{x}{1-x}$. Thus
\[H_2(t_v(\beta))'=-\log\frac{t_v}{1-t_v} t_v'(\beta)\] 

Differentiating, we get
\begin{align*}
\partial_\beta &\Phi(\beta,t(\beta))\\&=\frac{d}{k}\frac{-\sum_{i=0}^k {{k}\choose{i}} h_i\exp(-\beta h_i+it(\beta))}{\sum_{i=0}^k {{k}\choose{i}} \exp(-\beta h_i+it(\beta))}+t_v'(\beta)((d-1)\log\frac{t_v(\beta)}{1-t_v(\beta)}-dt(\beta))\\&=\frac{d}{k}\frac{-\sum_{i=0}^k {{k}\choose{i}} h_i\exp(-\beta h_i+it(\beta))}{\sum_{i=0}^k {{k}\choose{i}} \exp(-\beta h_i+it(\beta))}.
\end{align*}

From this, one can easily obtain the following theorem
\begin{lemma}\label{interpol}
Assume that for each $\beta$, we have a solution $t(\beta)$ of Equation \eqref{BPfixedpointEq}. Furthermore, assume that $t(\beta)$ is continuous and differentiable at all but finitely many points. Then for any $\beta_1\ge 0$, we have
\[\Phi(\beta_1,t(\beta_1))-\Phi(0,t(0))=\frac{d}{k}\int_0^{\beta_1} \frac{-\sum_{i=0}^k {{k}\choose{i}} h_i\exp(-\beta h_i+it(\beta))}{\sum_{i=0}^k {{k}\choose{i}} \exp(-\beta h_i+it(\beta))} d\beta.\]
\end{lemma}

\subsection{Correlation inequalities}

\begin{lemma}\label{FKG}
Let $X$ be a finite set, and let $\nu:X\to\mathbb{R}_{\ge 0}$ and $f,g:X\to \mathbb{R}$. Assume that for any $x,y\in X$, we have
\[f(x)\le f(y) \Rightarrow g(x)\le g(y).\]
Then
\[\left(\sum_{x\in X} \nu(x)f(x)\right)\left(\sum_{x\in X} \nu(x)g(x)\right)\le \left(\sum_{x\in X} \nu(x)\right)\left(\sum_{x\in X} \nu(x)f(x)g(x)\right).\]
Also, if  for any $x,y\in X$, we have
\[f(x)\le f(y) \Rightarrow g(x)\ge g(y).\]
Then
\[\left(\sum_{x\in X} \nu(x)f(x)\right)\left(\sum_{x\in X} \nu(x)g(x)\right)\ge \left(\sum_{x\in X} \nu(x)\right)\left(\sum_{x\in X} \nu(x)f(x)g(x)\right).\]
\end{lemma}
\begin{proof}
Consider a total ordering $\le$ of $X$ such  that $x\le y$ implies $f(x)\le f(y)$. Then, since every totally ordered set forms a distributive lattice, we can apply the FKG inequality \cite{fortuin1971correlation}, to get the statement. 
\end{proof}
We also need the following consequence.
\begin{lemma}\label{FKG2}
Let $X$ be a finite set, and let $\nu,f,g:X\to\mathbb{R}_{\ge 0}$ and $K:X\to \mathbb{R}$. Assume that for any $x,y\in X$, we have
\[K(x)\le K(y) \Rightarrow \frac{f(x)}{g(x)}\le \frac{f(y)}{g(y)}.\]
Then
\[\left(\sum_{x\in X} \nu(x)f(x)\right)\left(\sum_{x\in X} \nu(x)g(x) K(x)\right)\le \left(\sum_{x\in X} \nu(x)g(x)\right)\left(\sum_{x\in X} \nu(x)f(x)K(x)\right).\]
Also, if for any $x,y\in X$, we have
\[K(x)\le K(y) \Rightarrow \frac{f(x)}{g(x)}\ge \frac{f(y)}{g(y)}.\]
Then
\[\left(\sum_{x\in X} \nu(x)f(x)\right)\left(\sum_{x\in X} \nu(x)g(x) K(x)\right)\ge \left(\sum_{x\in X} \nu(x)g(x)\right)\left(\sum_{x\in X} \nu(x)f(x)K(x)\right).\]
\end{lemma} 
\begin{proof}
Apply the previous lemma for $\nu\cdot g$, $\frac{f}{g}$, $K$.
\end{proof}

For $v,w\in \{0,1\}^k$, we define $v\vee w,v\wedge w\in \{0,1\}^k$ as their component wise maximum and minimum, respectively. A function $f:\{0,1\}^k\to \mathbb{R}_{\ge 0}$ is called log-supermodular, if for all $v,w\in \{0,1\}^k$, we have
\[f(v)f(w)\le f(v\vee w)f(v\wedge w).\]

We will need the following consequence of the so called Four Function Theorem \cite{ahlswede1978inequality}.

\begin{lemma}\label{lsmmarg}
Let $f:\{0,1\}^k\to \mathbb{R}_{\ge 0}$ be a log-supermodular function. Then $f$ has log-supermodular marginals. In other words, let $\ell<k$, and define $g:\{0,1\}^\ell\to \mathbb{R}_{\ge 0}$ as
\[g(v_1,v_2,\dots,v_\ell)=\sum_{(v_{\ell+1},v_{\ell+2},\dots,v_k)\in \{0,1\}^{k-\ell}} f(v_1,v_2,\dots,v_k).\]
Then $g$ is log-supermodular.
\end{lemma}

\section{The proof of Lemma \ref{lemmafix}}\label{sec3}
For $v=(v_1,v_2,\dots,v_k)\in \{0,1\}^k$, we define
\[|v|=\sum_{i=1}^k v_i,\]
and 
\[H(v)=h_{|v|}.\]
For $v=(v_1,v_2,\dots,v_{k-1})\in \{0,1\}^{k-1}$, the vectors $(v_1,v_2,\dots,v_{k-1},0)$ and $(v_1,v_2,\dots,v_{k-1},1)$ are denoted by $v0$ and $v1$, respectively.

Observe that
\begin{align*}
F(\beta,t)&=-t+(d-1)\log\left(\frac{\sum_{i=0}^{k-1} {{k-1}\choose{i}} \exp(-\beta h_{i+1}+it)}{\sum_{i=0}^{k-1} {{k-1}\choose{i}} \exp(-\beta h_{i}+it)}\right)
\\&=-t+(d-1)\log \left(\frac{\sum_{v\in \{0,1\}^{k-1}} \exp\left(-\beta H(v1)+|v|t\right)}{\sum_{v\in \{0,1\}^{k-1}} \exp\left(-\beta H(v0)+|v|t\right)}\right).
\end{align*}

By the symmetry of the sequence $h_0,h_1,\dots, h_k$, we have  
\[F(\beta,-t)=-F(\beta,t).\]



From the observation that $F(\beta,-t)=-F(\beta,t)$, it is straightforward to see that $F(\beta,0)=0$ for any $\beta$. It also shows that we can restrict our attention to the case $t>0$. 

\begin{lemma}\label{lemmaconcave}
For any fixed $\beta\ge 0$, the function $t\mapsto F(\beta,t)$ is concave on $\mathbb{R}_{\ge 0}$. 
\end{lemma}
\begin{proof}
Let us introduce the function 
\begin{align*}
F_2(t_1,t_2,\dots, t_{k-1})=-\frac{1}{k-1}\sum_{i=1}^{k-1} t_i&+(d-1)\log\left(\sum_{v\in \{0,1\}^{k-1}} \exp\left(-\beta H(v1)+\sum_{i=1}^{k-1}v_i t_i\right)\right)\\&-(d-1)\log\left(\sum_{v\in \{0,1\}^{k-1}} \exp\left(-\beta H(v0)+\sum_{i=1}^{k-1}v_i t_i\right)\right)
\end{align*}

We claim that if $t_1=t_2=\dots=t_{k-1}=t$, then $\partial_{t_{k-1}}^2 F_2(t_1,t_2,\dots,t_{k-1})\le 0$. First, let us define
\begin{align*}
A&=A(t_1,t_2,\dots,t_{k-2})=\sum_{v\in \{0,1\}^{k-2}} \exp\left(-\beta H(v11)+\sum_{i=1}^{k-2} v_i t_i\right)\\
B&=B(t_1,t_2,\dots,t_{k-2})=\sum_{v\in \{0,1\}^{k-2}} \exp\left(-\beta H(v01)+\sum_{i=1}^{k-2} v_i t_i\right)\\
C&=C(t_1,t_2,\dots,t_{k-2})=\sum_{v\in \{0,1\}^{k-2}} \exp\left(-\beta H(v10)+\sum_{i=1}^{k-2} v_i t_i\right)\\
D&=D(t_1,t_2,\dots,t_{k-2})=\sum_{v\in \{0,1\}^{k-2}} \exp\left(-\beta H(v00)+\sum_{i=1}^{k-2} v_i t_i\right)
\end{align*}

It is straightforward from the definition of $H$ that we have $B=C$. 

With these notations, we have
\[F_2(t_1,t_2,\dots, t_{k-1})=-\frac{1}{k-1}\sum_{i=1}^{k-1} t_i+(d-1)\log\left(A\exp(t_{k-1})+B\right)-(d-1)\log\left(C\exp(t_{k-1})+D\right).\]  

Thus, a direct computation gives us that
\[\partial_{t_{k-1}}^2 F_2(t_1,t_2,\dots,t_{k-1})=\frac{(d-1)\exp(t_{k-1})(BC-AD)(AC\exp (2t_{k-1})-BD)}{\left(A\exp(t_{k-1})+B\right)^2 \left(C\exp(t_{k-1})+D\right)^2}.\]

We define the function $f$ on $\{0,1\}^k$ as
\[f(v)=\exp\left(-\beta H(v)+\sum_{i=1}^{k-2} v_i t_i\right).\]
Since $f$ is logsupermodular, it has logsupermodular marginals by Lemma \ref{lsmmarg}, in particular,
\[BC\le AD.\]

\begin{lemma}
If $(t_1,t_2,\dots,t_{k-2})\in \mathbb{R}_{\ge 0}^{k-2}$, then 
\[A\ge D.\]
\end{lemma}
\begin{proof}
Let us consider the function
\[R(t_1,t_2,\dots,t_{k-2})=\log A(t_1,t_2,\dots,t_{k-2})-\log D(t_1,t_2,\dots,t_{k-2}).\]
From the symmetry of the sequence $h_0,h_1,\dots,h_k$, we see that $R(0)=0$. Thus, the statement will follow, once we prove that all the partial derivatives of $R$ are non-negative on $\mathbb{R}_{\ge 0}^{k-1}$. For simplicity of notation, we will show this for $\partial_{t_{k-2}} R$.

First, let us define
\begin{align*}
A^\sharp &=A^\sharp(t_1,t_2,\dots,t_{k-2})=\sum_{v\in \{0,1\}^{k-3}} \exp\left(-\beta H(v111)+\sum_{i=1}^{k-3} v_i t_i\right)\\
B^\sharp &=B^\sharp(t_1,t_2,\dots,t_{k-2})=\sum_{v\in \{0,1\}^{k-3}} \exp\left(-\beta H(v011)+\sum_{i=1}^{k-3} v_i t_i\right)\\
C^\sharp &=C^\sharp(t_1,t_2,\dots,t_{k-2})=\sum_{v\in \{0,1\}^{k-3}} \exp\left(-\beta H(v100)+\sum_{i=1}^{k-3} v_i t_i\right)\\
D^\sharp &=D^\sharp(t_1,t_2,\dots,t_{k-2})=\sum_{v\in \{0,1\}^{k-3}} \exp\left(-\beta H(v000)+\sum_{i=1}^{k-3} v_i t_i\right)
\end{align*}

Then, we have
\[R=\log\left(A^\sharp\exp(t_{k-2})+B^\sharp\right)-\log\left(C^\sharp\exp(t_{k-2})+D^\sharp\right),\]
and 
\[\partial_{t_{k-2}} R=\frac{\exp(t_{k-2})(A^\sharp D^\sharp-B^\sharp C^\sharp)}{(A^\sharp\exp(t_{k-2})+B^\sharp)(C^\sharp\exp(t_{k-2})+D^\sharp)}\ge 0,\]
since the inequality $A^\sharp D^\sharp \ge B^\sharp C^\sharp$ can be proved like above.
\end{proof}

Combining the lemma above with the facts $B=C$ and $t_{k-2}>0$, we have
\[AC\exp(2t_{k-1})-BD\ge DC-BD=0.\]

Therefore, 
\[\partial_{t_{k-1}}^2 F_2(t_1,t_2,\dots,t_{k-1})=\frac{(d-1)\exp(t_{k-1})(BC-AD)(AC\exp(2t_{k-1})-BD))}{\left(A\exp(t_{k-1})+B\right)^2 \left(C\exp(t_{k-1})+D\right)^2}\le 0.\]

By symmetry, we have $\partial_{t_{i}}^2 F_2(t_1,t_2,\dots,t_{k-1})\le 0$ for any $1\le i\le k-1$. In particular, $\partial_{t_{i}} F_2(t_1,t_2,\dots,t_{k-1}) $ monotone decreasing. By the chain rule, we have
\[\partial_t F(\beta,t)=\sum_{i=1}^{k-1} \partial_{t_i} F_2(t,t,\dots,t)\]
is monotone decreasing. This concludes the proof of Lemma \ref{lemmaconcave}.
\end{proof}

\begin{lemma}\label{betacexists}
For any fixed $\beta\ge 0$, the function $t\mapsto F(\beta,t)$ has a unique positive zero, if and only if $\partial_t F(\beta,0)>0$. If $\partial_t F(\beta,0)\le 0$, then   $t\mapsto F(\beta,t)$ has no positive zero.
\end{lemma}
\begin{proof}
Observe that $F(\beta,0)=0$ and
\[\lim_{t\to \infty} F(\beta,t)=-\infty.\]
The statement follows by combining these observation with Lemma \ref{lemmaconcave}.
\end{proof}

\begin{lemma}\label{difft}
Consider a $\beta\ge 0$ such that $\partial_t F(\beta,0)>0$. Let $t^+$ be the unique positive zero of $t\mapsto F(\beta,t)$. Then $\partial_t F(\beta,t^+)<0$.
\end{lemma}
\begin{proof}
Let $f(t)=F(\beta,t)$. Since $f(0)=f(t^+)=0$ and $f'(0)>0$, we have a local maximum point $t_0$ of $f$ in the interval $(0,t)$. Then $f'(t_0)=0$ and $f(t_0)>f(t^+)$. By Lemma \ref{lemmaconcave}, $f$ is concave on $[0,\infty)$. It follows that $f'(t^+)<0$. 
\end{proof}

\begin{lemma}\label{monotonbeta}
The function $\beta\mapsto \partial_t F(\beta,0)$ is monotone increasing.
\end{lemma}
\begin{proof}

By the symmetry of the sequence $h_0,h_1,\dots,h_k$, we have
\[\sum_{v\in \{0,1\}^{k-1}} \exp(-\beta H(v1))=\sum_{v\in \{0,1\}^{k-1}} \exp(-\beta H(v0))=\frac{1}{2}\sum_{v\in \{0,1\}^{k}} \exp(-\beta H(v)).\]

By differentiating, we see that
\begin{align*}
\partial_t F(\beta,0)=-1+(d-1)\left(\frac{\sum_{v\in \{0,1\}^{k-1}} |v|\exp(-\beta H(v1))}{\sum_{v\in \{0,1\}^{k-1}} \exp(-\beta H(v1))}-\frac{\sum_{v\in \{0,1\}^{k-1}} |v|\exp(-\beta H(v0))}{\sum_{v\in \{0,1\}^{k-1}} \exp(-\beta H(v0))}\right).
\end{align*}

Here we have
\begin{align*}
&\frac{\sum_{v\in \{0,1\}^{k-1}} |v|\exp(-\beta H(v1))}{\sum_{v\in \{0,1\}^{k-1}} \exp(-\beta H(v1))}-\frac{\sum_{v\in \{0,1\}^{k-1}} |v|\exp(-\beta H(v0))}{\sum_{v\in \{0,1\}^{k-1}} \exp(-\beta H(v0))}\\
&\qquad= \frac{\sum_{v\in \{0,1\}^{k-1}} |v|\exp(-\beta H(v1))-\sum_{v\in \{0,1\}^{k-1}} |v|\exp(-\beta H(v0))}{\frac{1}{2}\sum_{v\in \{0,1\}^{k}} \exp(-\beta H(v))}.
\end{align*}

We set
\[\gamma=\sum_{v\in \{0,1\}^{k-1}} |v|\exp(-\beta H(v1))-\sum_{v\in \{0,1\}^{k-1}} |v|\exp(-\beta H(v0)).\]

Note that 
\[\gamma=\sum_{\substack{v\in \{0,1\}^{k}\\v_k=1}} (|v|-1)\exp(-\beta H(v))-\sum_{\substack{v\in \{0,1\}^{k}\\v_k=0}} |v|\exp(-\beta H(v)).\]

By symmetry, we have
\begin{align*}
k\gamma &=\sum_{i=1}^k\left(\sum_{\substack{v\in \{0,1\}^{k}\\v_i=1}} (|v|-1)\exp(-\beta H(v))-\sum_{\substack{v\in \{0,1\}^{k}\\v_i=0}} |v|\exp(-\beta H(v))\right)\\
&=\sum_{v\in \{0,1\}^k} \left(|v|(|v|-1)-(k-|v|)|v|\right) \exp(-\beta H(v))\\
&=\sum_{v\in \{0,1\}^k} \left(2\left(|v|-\frac{k}{2}\right)^2 -\frac{k^2}{2}+(k-1)|v|\right) \exp(-\beta H(v)).
\end{align*}

For $v=(v_1,v_2,\dots,v_{k})\in \{0,1\}^{k}$, we set $\bar{v}=(1-v_1,1-v_2,\dots,1-v_{k})\in \{0,1\}^{k}$. Now we have
\begin{align*}
\sum_{v\in \{0,1\}^k} |v|\exp(-\beta H(v))&=\frac{1}{2} \sum_{v\in \{0,1\}^k} (|v|\exp(-\beta H(v))+|\bar{v}|\exp(-\beta H(\bar{v}))\\&=\frac{1}{2} \sum_{v\in \{0,1\}^k} (|v|+|\bar{v}|)\exp(-\beta H(v))\\&=\frac{k}{2}\sum_{v\in \{0,1\}^k} \exp(-\beta H(v)).
\end{align*}

Therefore,
\[k\gamma=\sum_{v\in \{0,1\}^k} \left(2\left(|v|-\frac{k}{2}\right)^2 -\frac{k}{2}\right) \exp(-\beta H(v)).\]

We set $f(v)=2\left(|v|-\frac{k}{2}\right)^2 -\frac{k}{2}$. 
Now we have,
\[\partial_t F(\beta,0)=-1+\frac{2(d-1)}{k}\cdot\frac{\sum_{v\in \{0,1\}^k} f(v)\exp(-\beta H(v))}{\sum_{v\in \{0,1\}^k} \exp(-\beta H(v))}.\]

Taking derivative with respect to $\beta$, we have
\begin{align*} \partial_\beta \partial_t F(\beta,0)&=\frac{2(d-1)}{k} \cdot \frac{\left(\sum_{v\in \{0,1\}^k} f(v)\exp(-\beta H(v))\right)\left(\sum_{v\in \{0,1\}^k} H(v)\exp(-\beta H(v))\right)}{\left(\sum_{v\in \{0,1\}^k} \exp(-\beta H(v))\right)^2}\\&\qquad-\frac{2(d-1)}{k} \cdot\frac{\left(\sum_{v\in \{0,1\}^k} H(v)f(v)\exp(-\beta H(v))\right)\left(\sum_{v\in \{0,1\}^k} \exp(-\beta H(v))\right)}{\left(\sum_{v\in \{0,1\}^k} \exp(-\beta H(v))\right)^2}\\&\ge 0,
\end{align*}
by Lemma \ref{FKG}.
\end{proof}

\begin{lemma}\label{diffbeta}
Let $t>0$. For $\beta\ge 0$, we have $\partial_\beta F(\beta,t) \ge 0$.
\end{lemma}
\begin{proof}

We have

\[\sum_{v\in \{0,1\}^{k-1} } \exp(-H(v1)+|v|t)=\exp(-t)\sum_{\substack{v\in \{0,1\}^{k}\\ v_k=1} } \exp(-H(v)+|v|t).\]

By symmetry, we have
\begin{align*}
k\sum_{v\in \{0,1\}^{k-1} } &\exp(-\beta H(v1)+|v|t)\\&=\exp(-t)\sum_{i=1}^k \sum_{\substack{v\in \{0,1\}^{k}\\ v_i=1} } \exp(-\beta H(v)+|v|t)\\
&=\exp(-t) \sum_{v\in \{0,1\}^{k} } |v|\exp(-\beta H(v)+|v|t)\\
&=\frac{1}{2}\exp(-t) \sum_{v\in \{0,1\}^{k} } \left(|v|\exp(-\beta H(v)+|v|t) +|\bar{v}|\exp(-\beta H(\bar{v})+|\bar{v}|t) \right)\\
&=\frac{1}{2}\exp(-t) \sum_{v\in \{0,1\}^{k} } \exp(-\beta H(v))\left(|v|\exp(|v|t) +(k-|v|)\exp((k-|v|)t) \right).
\end{align*}

Similarly,
\begin{multline*}
k\sum_{v\in \{0,1\}^{k-1} } \exp(-\beta H(v0)+|v|t)\\=\frac{1}{2} \sum_{v\in \{0,1\}^{k} } \exp(-\beta H(v))\left((k-|v|)\exp(|v|t) +|v|\exp((k-|v|)t) \right).
\end{multline*}

Let us introduce the notations,
\begin{align*}
f(s)&=s\exp(st) +(k-s)\exp((k-s)t),\\
g(s)&=(k-s)\exp(st) +s\exp((k-s)t).
\end{align*}

\begin{lemma}\label{fperg}
We have
\[\frac{f(s)}{g(s)}=\frac{f(k-s)}{g(k-s)}.\]
If $t>0$, then the function $s\mapsto \frac{f(s)}{g(s)}$ is monotone decreasing on the interval $(-\infty,\frac{k}2]$, and 
monotone decreasing on the interval $[\frac{k}2,\infty)$.
\end{lemma}
\begin{proof}
It follows by the examination of the derivative
\[\partial_s \frac{f(s)}{g(s)}=-\frac{2k\exp(kt)\left(t(k-2s)+\sinh(t(k-2s))\right)}{\left((k-s)\exp(st) +s\exp((k-s)t)\right)^2}.\]

\end{proof}

We have
\begin{align*}
\partial_\beta &F(\beta,t)\\&=(d-1)\partial_\beta \log\frac{\sum_{v\in \{0,1\}^{k} } \exp(-\beta H(v))f(|v|)}{\sum_{v\in \{0,1\}^{k} } \exp(-\beta H(v))g(|v|)}\\&=(d-1)\frac{\left(\sum_{v\in \{0,1\}^{k} } \exp(-\beta H(v))f(|v|)\right)\left(\sum_{v\in \{0,1\}^{k} } \exp(-\beta H(v))g(|v|)H(|v|)\right)}{\left(\sum_{v\in \{0,1\}^{k} } \exp(-\beta H(v))f(|v|)\right)\left(\sum_{v\in \{0,1\}^{k} } \exp(-\beta H(v))g(|v|)\right)}\\
&\qquad -(d-1)\frac{\left(\sum_{v\in \{0,1\}^{k} } \exp(-\beta H(v))f(|v|)H(|v|)\right)\left(\sum_{v\in \{0,1\}^{k} } \exp(-\beta H(v))g(|v|)\right)}{\left(\sum_{v\in \{0,1\}^{k} } \exp(-\beta H(v))f(|v|)\right)\left(\sum_{v\in \{0,1\}^{k} } \exp(-\beta H(v))g(|v|)\right)}\ge 0,
\end{align*}
where in the last step we used Lemma \ref{FKG2}. The assumptions of Lemma~\ref{FKG2} were verified in Lemma~\ref{fperg}. 

\end{proof}

We define
\[\beta_c=\inf \{\beta\quad|\quad \partial_t F(\beta,0)>0\}.\]

Note that $F(0,t)=-t$, thus $\partial_t F(0,0)<0$. By continuity  of $\partial_t F(\beta,0)$, it follows that $\beta_c>0$ and $\partial_t F(\beta_c,0)=0$.  
It follows from Lemma \ref{betacexists} combined with Lemma \ref{monotonbeta} that $t\mapsto F(\beta,t)$ has a unique positive zero if and only if $\beta>\beta_c$. For $\beta>\beta_c$, we define $\mathfrak{g}(\beta)$ as the unique positive zero of $t\mapsto F(\beta,t)$. For $\beta>\beta_c$, we have $\partial_t F(\beta,\mathfrak{g}(\beta))<0$ by Lemma \ref{difft} and $\partial_\beta F(\beta,\mathfrak{g}(\beta))\ge 0$ by Lemma \ref{diffbeta}. By the implicit function theorem, $\mathfrak{g}$ is differentiable at $\beta$ with derivative $-\frac{\partial_\beta F(\beta,\mathfrak{g}(\beta))}{\partial_t F(\beta,\mathfrak{g}(\beta))}\ge 0$. Thus, it follows that $\mathfrak{g}$ is differentiable and monotone increasing. Since $\mathfrak{g}$ is nonnegative and monotone increasing $t_0=\lim_{\beta\to\beta_c} \mathfrak{g}(\beta)$ exist and $t_0\in \mathbb{R}_{\ge 0}$. Since the zero set of $F$ is a closed set, we see that $F(\beta_c,t_0)=0$. Then it follows that we must have $t_0=0$.
\section{Minimizing the expected energy} \label{sec4}

Fix a $\beta$, and let us define
\begin{align*}
m(t_1,t_2,\dots, t_{k-1})&=\log\left(\sum_{v\in \{0,1\}^{k-1}} \exp\left(-\beta H(v1)+\sum_{i=1}v_i t_i\right)\right)\\&-\log\left(\sum_{v\in \{0,1\}^{k-1}} \exp\left(-\beta H(v0)+\sum_{i=1}v_i t_i\right)\right).
\end{align*}

Let $(u_{i\to a}),(\hat{u}_{a\to i})$ be a BP fixed point on the $(d,k)$-biregular tree $\mathbb{T}$. By Equation \eqref{ueq1} and Equation \eqref{ueq2}, we have
\begin{equation}\label{genfix}
u_{i\to a}=\sum_{b\in \partial i\backslash a} m((u_{j\to b})_{j\in \partial b\backslash i}).
\end{equation}

\begin{lemma}\label{monotonmess}
The function $m(t_1,t_2,\dots, t_{k-1})$ is monotone increasing  in each variable. 
\end{lemma}
\begin{proof}
By symmetry, it is enough to prove that it is increasing in $t_{k-1}$. 
 First, let us define
\begin{align*}
A&=A(t_1,t_2,\dots,t_{k-2})=\sum_{v\in \{0,1\}^{k-2}} \exp\left(-\beta H(v11)+\sum_{i=1}^{k-2} v_i t_i\right)\\
B&=B(t_1,t_2,\dots,t_{k-2})=\sum_{v\in \{0,1\}^{k-2}} \exp\left(-\beta H(v01)+\sum_{i=1}^{k-2} v_i t_i\right)\\
C&=C(t_1,t_2,\dots,t_{k-2})=\sum_{v\in \{0,1\}^{k-2}} \exp\left(-\beta H(v10)+\sum_{i=1}^{k-2} v_i t_i\right)\\
D&=D(t_1,t_2,\dots,t_{k-2})=\sum_{v\in \{0,1\}^{k-2}} \exp\left(-\beta H(v00)+\sum_{i=1}^{k-2} v_i t_i\right)
\end{align*}


With these notations, we have
\[m(t_1,t_2,\dots, t_{k-1})=\log\left(A\exp(t_{k-1})+B\right)-(d-1)\log\left(C\exp(t_{k-1})+D\right).\]  

Thus, a direct computation gives us that
\[\partial_{t_{k-1}} m(t_1,t_2,\dots,t_{k-1})=\frac{\exp(t_{k-1})(AD-BC)}{\left(A\exp(t_{k-1})+B\right) \left(C\exp(t_{k-1})+D\right)}.\]

As in the proof of Lemma \ref{lemmaconcave}, we have
\[BC\le AD,\]
which gives the statement.
\end{proof}

\begin{lemma}\label{lemmahatar}
Let $(u_{i\to a}),(\hat{u}_{a\to i})$ be a BP fixed point on the $(d,k)$-biregular tree $\mathbb{T}$. Then, for each edge $ia$, we have
\[-\mathfrak{g}(\beta)\le u_{i\to a}\le \mathfrak{g}(\beta). \]
\end{lemma}
\begin{proof}
By examining the formula \eqref{genfix}, we see that
 \[u_{i\to a}\le (d-1)\beta(h_{\lfloor k/2\rfloor}-h_0).\]
We set $t_0=(d-1)\beta(h_{\lfloor k/2\rfloor}-h_0)$, and for $n\ge 0$, we set
\[t_{n+1}=(d-1)m(t_n,t_n,\dots,t_n).\]
Using Lemma \ref{monotonmess},  by induction, we see that $u_{i\to a}\le t_{n}$ for all $n$. Thus,
\[u_{i\to a}\le\lim_{n\to\infty} t_n=\mathfrak{g}(\beta).\]
The inequality, $-\mathfrak{g}(\beta)\le u_{i\to a}$ can be proved the same way. 
\end{proof}

Fix a $\beta\ge 0$. For $t\in \mathbb{R}^k$, we define
\[a_\beta(t)=\frac{\sum_{v\in \{0,1\}^k} -H(v)\exp(-\beta H(v)+(v,t))}{\sum_{v\in \{0,1\}^k} \exp(-\beta H(v)+(v,t))},\] 
where $(v,t)=\sum_{i=1}^k v_it_i$.
Let $t^*=(\mathfrak{g}(\beta),\mathfrak{g}(\beta),\dots,\mathfrak{g}(\beta))\in \mathbb{R}^k$.

\begin{lemma}\label{lemmaabeta}
We have
\[\max_{t\in [-\mathfrak{g}(\beta),\mathfrak{g}(\beta)]^k} a_\beta(t)=a_\beta(t^*)=a_\beta(-t^*).\]
\end{lemma} 
\begin{proof}
Let $z=\exp(\mathfrak{g}(\beta))$. Fix any $t\in [-\mathfrak{g}(\beta),\mathfrak{g}(\beta)]^k$. For $i\in \{0,1,\dots, k\}$, we introduce the notations
\begin{align*}
e(i)&=\sum_{\substack{v\in \{0,1 \}^k\\ |v|=i}} \exp((v,t)),\\
f(i)&=\frac{e(i)}{{{k}\choose{i}}},\\
N_i&=z^{i}+z^{k-i}.
\end{align*}

\begin{lemma}\label{lemmamonN}
For $0\le i\le \lfloor \frac{k}2\rfloor$, the function
\[\frac{f(i)+f(k-i)}{z^{i}+z^{k-i}}=\frac{f(i)+f(k-i)}{N_i}\]
is monotone increasing.
\end{lemma}
\begin{proof}
It follows from Newton's inequality \cite[Theorem 144.]{hardy1952inequalities} that $\frac{f(i+1)}{f(i)}$ is monotone decreasing. Also note that
\[\frac{f(1)}{f(0)}=\frac{\sum_{i=1}^k \exp(t_i)}{k}\le z\]
and
\[\frac{f(k-1)}{f(k)}=\frac{\sum_{i=1}^k \exp(-t_i)}{k}\ge z^{-1}.\]
Thus,
\begin{equation}\label{korlat}
z\ge \frac{f(1)}{f(0)}\ge \frac{f(i+1)}{f(i)}\ge \frac{f(k-1)}{f(k)}\ge z^{-1}. 
\end{equation}

We need to prove that for $i<\frac{k}{2}$, we have
\[\frac{f(i+1)+f(k-i-1)}{N_{i+1}}\ge \frac{f(i)+f(k-i)}{N_{i}}.\]
Using the notations $\alpha=\frac{f(i+1)}{f(i)}$ and $\gamma=\frac{f(k-i)}{f(k-i-1)}$, this is equivalent to
\[\alpha N_i-N_{i+1}+\frac{f(k-i-1)}{f(i)}\left(N_i-\gamma N_{i+1}\right)\ge 0.\]

We distinguish two cases. 
\begin{itemize}
\item First case: $N_i-\gamma N_{i+1}\le 0$. 

Since $\frac{f(j+1)}{f(j)}$ is monotone decreasing, we have
\[\frac{f(k-i-1)}{f(i)}\le \alpha^{k-2i-1}.\] 
Also note $\alpha\ge \gamma$. Thus,
\begin{align*}
\alpha N_i-N_{i+1}+&\frac{f(k-i-1)}{f(i)}\left(N_i-\gamma N_{i+1}\right)\\&\ge
\alpha N_i-N_{i+1}+\alpha^{k-2i-1} \left(N_i-\gamma N_{i+1}\right)\\&\ge
\alpha N_i-N_{i+1}+\alpha^{k-2i-1} \left(N_i-\alpha N_{i+1}\right)\\
&=N_i(\alpha+\alpha^{k-2i-1})-N_{i+1}(1+\alpha^{k-2i}).
\end{align*}
Therefore, it is enough to prove that
\[\frac{N_{i+1}}{N_{i}}\le \frac{\alpha+\alpha^{k-2i-1}}{1+\alpha^{k-2i}}.\]
\begin{lemma}
Let us consider the function 
\[q(x)=\frac{x+x^{k-2i-1}}{1+x^{k-2i}}.\]
Then for any $x>0$, we have $q(x)=q(x^{-1})$.

Moreover, $q(x)$ is monotone increasing on $(0,1]$, and it is monotone decreasing on $[1,\infty)$.

Consequently, for $z\ge 1$, we have

\[\min_{x\in [z^{-1},z]} q(x)=q(z)=q(z^{-1}).\]  
\end{lemma}
\begin{proof}
The first statement can be verified by a direct calculation. 

The second statement can be obtained by investigating the derivative
\[q'(x)=\frac{(k-2i-1)(x^{k-2i}-x^{k-2i+2})+x^2-x^{2(k-2i)}}{x^2(1+x^{k-2i})^2}.\]
Note that here we need to use that $2i<k$.

The third statement follows by combining the previous two. 
\end{proof}

Recall our observation in \eqref{korlat} that $z^{-1}\le \alpha\le z$. Thus combining this with the previous lemma, we get
\[\frac{\alpha+\alpha^{k-2i-1}}{1+\alpha^{k-2i}}=q(\alpha)\ge \min_{x\in [z^{-1},z]} q(x)=q(z)=\frac{N_{i+1}}{N_i}.\]
Thus the statement follows.
\item Second case: $N_i-\gamma N_{i+1}\ge 0$. 

Since $\frac{f(j+1)}{f(j)}$ is monotone decreasing, we have
\[\frac{f(k-i-1)}{f(i)}\ge \gamma^{k-2i-1}.\] 
Also note $\alpha\ge \gamma$. Thus,
\begin{align*}
\alpha N_i-N_{i+1}+&\frac{f(k-i-1)}{f(i)}\left(N_i-\gamma N_{i+1}\right)\\&\ge
\gamma N_i-N_{i+1}+\gamma^{k-2i-1} \left(N_i-\gamma N_{i+1}\right)\\
&=N_i(\gamma+\gamma^{k-2i-1})-N_{i+1}(1+\gamma^{k-2i}).
\end{align*}
We  can be finish the proof the same way as we did in the previous case.
\end{itemize}

Now we are ready to finish the proof of Lemma \ref{lemmaabeta}.

Observe that 
\begin{align*}
a_\beta(t)&=\frac{\sum_{v\in \{0,1\}^k} -H(v)\exp(-\beta H(v)+(v,t))}{\sum_{v\in \{0,1\}^k} \exp(-\beta H(v)+(v,t))}\\&=\frac{\sum_{i=0}^k -h_i \exp(-\beta h_i) e(i)}{\sum_{i=0}^k  \exp(-\beta h_i) e(i)}\\
&=\frac{\sum_{i=0}^k -h_i \exp(-\beta h_i) (e(i)+e(k-i))}{\sum_{i=0}^k  \exp(-\beta h_i) (e(i)+e(k-i))}.
\end{align*}  
\end{proof}
In particular,
\[a_\beta(t^*)=\frac{\sum_{i=0}^k -h_i{{k}\choose{i}} (z^i+z^{k-i})\exp(-\beta h_i)}{\sum_{i=0}^k {{k}\choose{i}} (z^i+z^{k-i})\exp(-\beta h_i)}.\]
Thus, the statement of the theorem follows from Lemma~\ref{FKG2} with the choice of $X=\{0,1,\dots,k\}$, $\nu(i)=\exp(-\beta h_i)$, $f(i)= (e(i)+e(k-i))$, $g(i)={{k}\choose{i}}(z^i+z^{k-i})$, $K(i)=h_i$. The conditions Lemma~\ref{FKG2} are verified in Lemma \ref{lemmamonN}.

\end{proof}
Now we prove Lemma \ref{lemma4} in slightly different form.
\begin{lemma}\label{lemma22}
For a fixed $\beta$, let $\mu$ be a Gibbs measure on $\mathbb{T}$. Fix a factor node $a$. Let $X$ be a random element of $\{0,1\}^V$ with law $\mu$. Then
\[\mathbb{E}\left(-h\left(\sum_{i\in\partial a} X_i\right)\right) \le a_\beta(\mathfrak{g}(\beta),\mathfrak{g}(\beta),\dots,\mathfrak{g}(\beta)).\]
\end{lemma}
\begin{proof}
Since the set of Gibbs-measures form a simplex \cite[Chapter 7.3]{georgii2011gibbs}, we may assume that that $\mu$ is extremal. Let $(u_{i\to a}),(\hat{u}_{a\to i})$ be the corresponding BP fixed point provided by Lemma~\ref{lemma5}. This lemma also gives that
\[\mathbb{E}\left(-h\left(\sum_{i\in\partial a} X_i\right)\right)=a_\beta((u_{i\to a})_{i\in\partial a}).\]
By Lemma \ref{lemmahatar}, we have $u_{i\to a}\in [-\mathfrak{g}(\beta),\mathfrak{g}(\beta))]$. Thus by Lemma \ref{lemmaabeta}, we have
\[\mathbb{E}\left(-h\left(\sum_{i\in\partial a} X_i\right)\right)=a_\beta((u_{i\to a})_{i\in\partial a})\le a_\beta(\mathfrak{g}(\beta),\mathfrak{g}(\beta),\dots,\mathfrak{g}(\beta)).\]

\end{proof}

\section{Putting everything together: The proof of Theorem \ref{TheoremFO}}\label{putt}

For a $(d,k)$-biregular factor graph $G$ and a $\beta\ge 0$, we define the Gibbs probability measure $\mu_G^\beta$ on $\{0,1\}^V$ by setting
\[\mu_G^\beta(x)=\frac{\exp(-\beta H_G(x))}{Z_G(\beta)}.\] 

Note that 
\[\Phi_G'(\beta)=\frac{1}{|V|}\sum_{x\in \{0,1\}^V}\frac{-H_G(x)\exp(-\beta H_G(x))}{Z_G(\beta)}.\]
This can be also expressed as follows. Let $X\in \{0,1\}^V$ a random vector with law $\mu_G^\beta$, and let $a$ be a uniform random factor node independent from $X$. Then
\begin{equation}\label{phiGexp}
\Phi_G'(\beta)=\frac{d}{k}\mathbb{E}\left(-h\left(\sum_{i\in \partial a} X_i\right)\right),
\end{equation}
where the expectation is over the random choice of $X$ and $a$. 

By Fatou's lemma, we have
\[\limsup_{n\to\infty} (\Phi_{G_n}(\beta_1)-\Phi_{G_n}(0))\le \int_0^{\beta_1} \limsup_{n\to\infty} \Phi_{G_n}'(\beta) d\beta.\]
Now our aim is to give an upper bound on  $\limsup_{n\to\infty} \Phi_{G_n}'(\beta)$.  Recall the formula \eqref{phiGexp} for $\Phi_{G_n}'(\beta)$. Taking an appropriate subsequential weak limit one can obtain a Gibbs measure $\mu$ on the $(d,k)$-biregular infinite tree $\mathbb{T}$ with the following property. Let $a$ be a fixed variable node of $\mathbb{T}$, let $X\in \{0,1\}^V$ a random vector with law $\mu$, then
\[\limsup_{n\to\infty} \Phi_{G_n}'(\beta)=\frac{d}{k}\mathbb{E}\left(-h\left(\sum_{i\in \partial a} X_i\right)\right),\] 
where the expectation over a random choice of $X$. We do not give more details of this weak limit argument, we refer to \cite{sly2014counting} instead, where the  same idea was applied.

Using Lemma \ref{lemma22}, we obtain that
\[\limsup_{n\to\infty} \Phi_{G_n}'(\beta)=\frac{d}{k}\mathbb{E}\left(-h\left(\sum_{i\in \partial a} X_i\right)\right)\le \frac{d}{k} \frac{-\sum_{i=0}^k {{k}\choose{i}} h_i\exp(-\beta h_i+i\mathfrak{g}(\beta))}{\sum_{i=0}^k {{k}\choose{i}} \exp(-\beta h_i+i\mathfrak{g}(\beta))}
.\] 
Therefore, by Lemma \ref{interpol}, we have 
\[\limsup_{n\to\infty} (\Phi_{G_n}(\beta_1)-\Phi_{G_n}(0))\le \int_0^{\beta_1} \frac{d}{k} \frac{-\sum_{i=0}^k {{k}\choose{i}} h_i\exp(-\beta h_i+i\mathfrak{g}(\beta))}{\sum_{i=0}^k {{k}\choose{i}} \exp(-\beta h_i+i\mathfrak{g}(\beta))} d\beta=\Phi(\beta_1,\mathfrak{g}(\beta_1))-\Phi(0,g(0)).\]

Now it is easy to check that $\Phi_{G_n}(0)=\Phi(0,\mathfrak{g}(0))=\log 2$. Therefore,
\[\limsup_{n\to\infty} \Phi_{G_n}(\beta_1)\le \Phi(\beta_1,\mathfrak{g}(\beta_1)).\]

Since the sequence $h_i$ is concave, our graphical models fall into the category of log-supermodular graphical models. Therefore the results of Ruozzi \cite{ruozzi2012bethe} can be applied to obtain that 
\[\Phi_{G_n}(\beta_1)\ge \Phi(\beta_1,\tau)\]
for any $\tau$ in the local marginal polytope of $G$. Choosing $\tau$ by using the fixed point $\mathfrak{g}(\beta_1)$ as it was done in Subsection \ref{subsecbethe}, we obtain that
\[\Phi_{G_n}(\beta_1)\ge \Phi(\beta_1,\mathfrak{g}(\beta_1)).\]
Consequently,
\[\liminf_{n\to\infty} \Phi_{G_n}(\beta_1)\ge \Phi(\beta_1,\mathfrak{g}(\beta_1)).\]
Thus, the statement follows.     
\bibliography{references}
\bibliographystyle{plain}

\Addresses

\appendix
\section{Extremal Gibbs measures and BP fixed points: The proof of Lemma \ref{lemma5}}

Let $\nu$ be an extremal Gibbs-measure, and let $X$ be a random element of $\{0,1\}^V$ with law $\nu$. For a subset $U$ of $V$, let $\mathcal{F}_U$ be the sigma-algebra generated by $(X_i)_{i\in U}$. Given an edge $ia$, we define $U(ia)\subset V$ such that $j\in U(ia)$ if and only if the unique path from $j$ to $a$ in $\mathbb{T}$ contains~$i$. 

We say that $\nu$ is a Markov-chain if for any edge $ia$ and $x_{\partial a\backslash i}\in \{0,1\}^{\partial a\backslash i}$, we have
\[\mathbb{P}(X_{\partial a\backslash i}=x_{\partial a\backslash i}|\mathcal{F}_{U(ia)})=\mathbb{P}(X_{\partial a\backslash i}=x_{\partial a\backslash i}|\mathcal{F}_{\{i\}}).\] 

\begin{lemma}
If $\nu$ is an extremal Gibbs measure, then it is a Markov-chain.
\end{lemma}
\begin{proof}
It can be proved along the lines of \cite[Theorem 12.6]{georgii2011gibbs}.
\end{proof}
Given an edge $ia$ and $x_i\in\{0,1\}, x_{\partial a\backslash i}\in \{0,1\}^{\partial a\backslash i}$, we define
\[P_{ia}(x_i,x_{\partial a\backslash i})=\frac{\mathbb{P}(X_{\partial a} =x_{\partial a})}{\mathbb{P}(X_i=x_i)}.\]

Let us consider $V_R$ and $F_R$ us in the statement of the lemma. Take any $j\in V_R$. Orient the edges of $\mathbb{T}$ such that they are oriented away from $j$. For $a\in F$, let $i(a)$ be the unique neighbor of $a$ such that the edge $i(a)a$ is oriented towards $a$. This definition is consistent with the previous definition of $i(a)$. Then it can be proved by induction that for $x_{V_R}\in \{0,1\}^{V_R}$, we have
\begin{equation}\label{formula1}
\mathbb{P}(X_{V_R}=x_{V_R})=\mathbb{P}(X_j=x_j)\prod_{a\in F_R} P_{i(a)a}(x_i(a),x_{\partial a\backslash i(a)}).
\end{equation}   

Now set $F_R'=F_R\cup \partial R$, $V_R'=\partial F_R'$. Let $H=V_R'\backslash V_R$. 

We have
\[\mathbb{P}(X_{V_R}=x_{V_R})=\frac{\mathbb{P}(X_{V_R}=x_{V_R})}{\mathbb{P}(X_{V_R}=x_{V_R},X_H=0)}\cdot \frac{\mathbb{P}(X_{V_R}=x_{V_R},X_H=0)}{\mathbb{P}(X_{V_R}=0,X_H=0)}\cdot \mathbb{P}(X_{V_R}=0,X_H=0).\] 

Using equation \eqref{formula1}, we see that
\[\frac{\mathbb{P}(X_{V_R}=x_{V_R})}{\mathbb{P}(X_{V_R}=x_{V_R},X_H=0)}=\left(\prod_{a\in \partial R} P_{i(a)a}(x_{i(a)},0)\right)^{-1}.\]

From the fact that $\nu$ is a Gibbs-measure, we have

\[\frac{\mathbb{P}(X_{V_R}=x_{V_R},X_H=0)}{\mathbb{P}(X_{V_R}=0,X_H=0)}=\frac{\mathbb{P}(X_{V_R}=x_{V_R}|X_H=0)}{\mathbb{P}(X_{V_R}=0|X_H=0)}=\prod_{a\in F_R}\frac{\psi_a(x_{\partial a})}{\psi_a(0)}\prod_{a\in \partial R} \frac{\psi_a(x_{i(a)}0)}{\psi_a(0)},\]
where $x_{i(a)}0$ stands for the vector $y_{\partial a}\in \{0,1\}^{\partial a}$ such that $y_{i(a)}=x_{i(a)}$ and $y_j=0$ for all $j\in \partial a\backslash i(a)$.

By setting
\[Z=\frac{\prod_{a\in F_R'}\psi_a(0)}{\mathbb{P} (X_{V_R}=0,X_H=0)}\]
and \[\hat{\ell}_{a\to i}(x_i) =\frac{\psi_a(x_i0)}{P_{ia}(x_i ,0)},\]
we have
\[\mathbb{P}(X_{V_R}=x_{V_R})=\frac{1}{Z}\prod_{a\in F_R} \psi_a(x_{\partial a}) \prod_{a\in \partial R} \hat{\ell}_{a\to i(a)}(x_{i(a)}).\]

By setting $\hat{u}_{a\to i}=\log \frac{\hat{\ell}_{a\to i}(1)}{\hat{\ell}_{a\to i}(0)}$, we have
\[\mathbb{P}(X_{V_R}=x_{V_R})=\frac{1}{Z'}\prod_{a\in F_R} \psi_a(x_{\partial a}) \exp\left(\sum_{a\in \partial R} x_{i(a)} \hat{u}_{a\to i(a)}\right).\]
We can also set
\[u_{i\to a}=\sum_{b\in \partial i\backslash a} \hat{u}_{b\to i}.\]
It remains to prove that $(u_{i\to a}),(\hat{u}_{a\to i})$ is indeed a BP fixed point.

Set $R=\{a\}\cup \partial a$. We have
\begin{align*}
\hat{\ell}_{a\to i}(x_i)&=\frac{\psi_a(x_i0)}{P_{ia}(x_i ,0)}
\\&=\frac{\psi_a(x_i0)\mathbb{P}(X_i=x_i)}{\mathbb{P}(X_{\partial a}=x_i0)}
\\&=\frac{\psi_a(x_i0)\sum_{x_{\partial a\backslash i}}\mathbb{P}(X_{\partial a}=x_{\partial a})}{\mathbb{P}(X_{\partial a}=x_i0)}
\\&=\frac{\psi_a(x_i0)\sum_{x_{\partial a\backslash i}}Z^{-1} \psi_a(x_{\partial a}) \prod_{b\in \partial R} \hat{\ell}_{b\to i(b)}(x_{i(b)})}{Z^{-1}\psi_a(x_i0)\prod_{b\in \partial i\backslash a}\hat{\ell}_{b\to i}(x_i) \prod_{b\in \partial R\backslash \partial i}\hat{\ell}_{b\to i(b)}(0)} 
\\&=\sum_{x_{\partial a\backslash i}}\psi_a(x_{\partial a}) \prod_{b\in \partial R\backslash \partial i} \frac{  \hat{\ell}_{b\to i(a)}(x_{i(a)})}{\hat{\ell}_{b\to i(b)}(0)}
\\&=\sum_{x_{\partial a\backslash i}}\psi_a(x_{\partial a}) \exp\left(\sum_{b\in \partial R\backslash \partial i} x_{i(b)}\hat{u}_{b\to i(b)}\right).
\end{align*}

Therefore, 
\begin{align*}\hat{u}_{a\to i}&=\log\frac{\sum_{x_{\partial a}, x_i=1 } \psi_a(x_{\partial_a})  \exp\left(\sum_{b\in \partial R\backslash \partial i} x_{i(b)}\hat{u}_{b\to i(b)}\right)}{\sum_{x_{\partial a}, x_i=0 } \psi_a(x_{\partial_a})  \exp\left(\sum_{b\in \partial R\backslash \partial i} x_{i(b)}\hat{u}_{b\to i(b)}\right)}
\\&=\log\frac{\sum_{x_{\partial a}, x_i=1 } \psi_a(x_{\partial_a}) \exp\left(\sum_{j\in \partial a-i} x_j u_{j\to a}\right)}{\sum_{x_{\partial a}, x_i=0 } \psi_a(x_{\partial_a}) \exp\left(\sum_{j\in \partial a-i} x_j u_{j\to a}\right)}.
\end{align*}

\end{document}